\documentclass{amsproc}

\usepackage{amsmath, amsthm}%
\usepackage{amsfonts}%
\usepackage{amssymb}%
\usepackage{graphicx}
\usepackage{colonequals}
\usepackage{amscd, amstext, hyperref}
\usepackage{pictexwd}
\usepackage{dcpic}
\usepackage[mathscr]{eucal}
\usepackage{multicol}
\usepackage{multirow}
\usepackage{color}
\usepackage{bm}
\usepackage{comment}
\usepackage{subfigure}
\usepackage{graphicx}

\usepackage{tkz-euclide}
\usetikzlibrary{shapes.callouts}
\usetikzlibrary{arrows}
\tikzset{
  connect/.style = { dashed, red },
  connect1/.style = { dashed, blue },
  notice/.style  = { draw, rectangle callout, callout relative pointer={#1} },
  label/.style   = {text width=2cm }
  point/.style = {circle,fill,inner sep=1.5pt},
}
\tikzstyle{line} = [line width=0.2mm,draw, -latex]

\usepackage{tikz,tikz-3dplot}

\tdplotsetmaincoords{70}{165}

\newtheorem{theorem}{Theorem}[section]
\newtheorem{lemma}[theorem]{Lemma}
\newtheorem{proposition}[theorem]{Proposition}
\newtheorem{corollary}[theorem]{Corollary}

\newtheorem*{maintheorem*}{Main Theorem}
\newtheorem*{corollary*}{Corollary}

\newtheorem*{lemma*}{lemma}

\theoremstyle{definition}
\newtheorem{definition}[theorem]{Definition}
\newtheorem{example}[theorem]{Example}

\theoremstyle{prop}

\theoremstyle{remark}
\newtheorem{remark}[theorem]{Remark}

\numberwithin{equation}{section}
\newcommand{\Z}{\mathbb{Z}}

\newcommand{\R}{\mathbb{R}}
\newcommand{\C}{\mathbb{C}}

\newcommand{\PP}{\mathbb{P}}

\begin{document}

\title{Higher dimensional origami constructions}

\author{Deveena Banerjee}
\address{Department of Molecular Physiology \& Biophysics\\
Vanderbilt University\\
2400 Highland Ave\\
Nashville, TN 37212
}
\email{deveena.r.banerjee@vanderbilt.edu}
\author{Sara Chari}
\address{Department of Mathematics\\
Bates College\\
3 Andrews Rd\\
Lewiston, ME 04240}
\email{schari@bates.edu}
\author{Adriana Salerno}
\email{asalerno@bates.edu}

\begin{abstract}
Origami is an ancient art that continues to yield both artistic and scientific insights to this day. In 2012, Buhler, Butler, de Launey, and Graham extended these ideas even further by developing a mathematical construction inspired by origami -- one in which we iteratively construct points on the complex plane (the ``paper") from a set of starting points (or ``seed points") and lines through those points with prescribed angles (or the allowable ``folds" on our paper). Any two lines with these prescribed angles through the seed points that intersect generate a new point, and by iterating this process for each pair of points formed, we generate a subset of the complex plane. We extend previously known results about the algebraic and geometric structure of these sets to higher dimensions. In the case when the set obtained is a lattice, we explore the relationship between the set of angles and the generators of the lattice and determine how introducing a new angle alters the lattice.
\end{abstract}

\maketitle

\section{Introduction}

Origami, from the Japanese words for \emph{fold} (\emph{oru}) and \emph{paper} (\emph{kami}) is an ancient art that continues to yield both artistic and scientific insights to this day. In \cite{BBDLG}, Buhler et al. expand these horizons even further by developing a mathematical construction inspired by origami -- one in which we iteratively construct points on the complex plane (the ``paper") from a set of starting points (or ``seed points") and lines through those points with prescribed angles (or the allowable ``folds" on our paper). The questions they studied in their paper, and which many mathematicians later pursued, were natural and deep: when does one have an origami construction with a given structure? In the complex plane, these questions can be algebraic (when does the construction yield a ring?) or geometric (when does the construction yield a lattice? a dense subset?). In this paper, we extend some of these results to higher-dimensional vector spaces and algebras, showing that these ideas extend beyond the constraints of a two dimensional ``paper". We focus our attention on lattices.

%Before stating the new results, we will review some terminology and definitions from prior work, using the notation established in \cite{Moller}. 

Let $S$ be the unit circle in $\C$. We define the set of allowable \emph{directions}  to be a subset   $U$ of  $S/\{\pm1\}$ containing $1$. Let $p$ and $q$ be points in $\C$. Then, the line from point $p$ with slope given by $\alpha\in U$ is given by the set of points $p+r\alpha$, where $r \in \R$. Similarly, the line through $q$ at an angle $\beta\in U$ is the set of points of the form $q+s\beta$ for $s \in \R$. The intersection of the line through $p$ in the direction of $\alpha$ and the line through $q$ in the direction of $\beta,$ will be denoted by $$[\![p,q]\!]_{\alpha, \beta},$$ and is given by the unique point
\begin{equation}\label{intersection}
z=p+r\alpha=q+s\beta.
\end{equation} 

The construction of such an intersection point $z=[\![p,q]\!]_{\alpha, \beta}$ can be seen in Figure \ref{basicintersection}.

\begin{figure}[!h]
\centering
\includegraphics[scale=1]{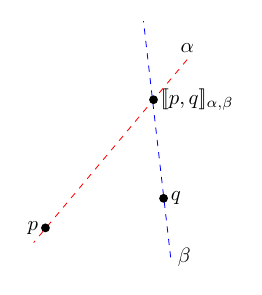}
\caption{Construction of the point $z=[\![p,q]\!]_{\alpha, \beta}$ by extending along $\alpha$ from $p$ and $\beta$ from $q$ until the lines intersect.}
\label{basicintersection}
\end{figure}

One advantage of working over the complex numbers, rather than just $\R^2,$ is that we can use the ring operations of $\C.$ So, for instance, one can write the intersection in a way that is easier to compute, and we obtain the so-called \emph{intersection operator} in \cite{BBDLG}. 

\begin{equation}
[\![ p,q ]\!]_{\alpha, \beta}= \frac{\alpha \overline{p} - \overline{\alpha} p}{\alpha \overline{\beta} - \overline{\alpha}\beta} \beta + \frac{\beta \overline{q} - \overline{\beta} q}{\overline{\alpha} \beta - \alpha \overline{\beta}} \alpha. \label{equation1}
\end{equation}

\begin{remark}
Notice that there are many ways to define $U$. For example, one could consider a subset of $[0,\pi)$ containing 0, determining the allowable slopes of our lines in the construction.  One could then use polar coordinates to identify the corresponding points on the unit circle, and thus we could generalize using $n$-dimensional spherical coordinates \cite{Blumenson}. 

We could also have defined $U$ to be a subset of $\R\PP^1$ containing $[1\colon0]$. Thus, by way of the spherical model of projective space, we can identify every element of $U$ with a point in the unit circle and where antipodal points are identified. By abusing notation, we can then think of elements in $U$ as complex numbers, by identifying $[a\colon b]=a+bi$.

Our current choice was made due to the simplicity of generalization to $n$-dimensional space, and to match the notation of previous results.  We will make no distinction hereafter between the terms ``direction", ``slope", and ``angle". 
%\as{Add comment about $n$-spherical coordinates.} 
\end{remark}

%if it exists. If no such point exists, then we write $[\![p,q]\!]_{\alpha, \beta} \colonequals \emptyset.$

An \emph{origami construction} (or \emph{origami set}) is defined recursively as follows (here we use notation due to M\"{o}ller \cite{Moller}). 

\begin{definition}\label{origami}
Let $M_0=\{0,1\}$, which we call the set of \emph{seed points}. For $k\in\mathbb{N}$, define $M_k$ to be the set of all intersection points of lines through points in $M_{k-1}$ with slopes in $U$. So, $$M_k:=\{[\![p,q]\!]_{\alpha, \beta} \mid p,q \in M_{k-1} \text{ and } \alpha, \beta\in U\}.$$
The union $$M(U)\colonequals \bigcup_{k=0}^{\infty}M_k$$ is called the \emph{origami construction} with respect to the slopes given by $U$. 
\end{definition}

\begin{remark}
Even though $M_0$ will remain fixed throughout this paper, it is potentially interesting to change this set (as suggested in \cite{REU}). 
\end{remark}

\begin{example}\label{Gauss}

Recall that we always start with the seed points $M_0=\{0,1\}$. Let $U= \left\{1, i, \frac{1}{\sqrt{2}}+\frac{1}{\sqrt{2}}i\right\}$ be the set of permissible directions. Figure \ref{fig:figure} shows the progression at each stage, and Figure \ref{gauss} shows that the resulting set is actually $\Z[i]$, the Gaussian integers. 

\begin{figure}
\centering
\subfigure[]{%
\includegraphics[scale=1]{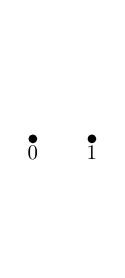}}
\subfigure[]{%
  \includegraphics[scale=1]{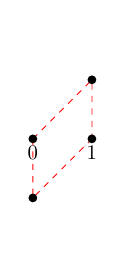}}
\subfigure[]{%
  \includegraphics[scale=1]{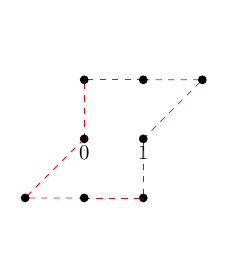}}
\subfigure[]{%
  \includegraphics[scale=1]{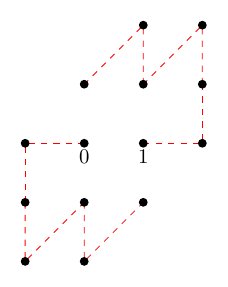}}
\caption{(a) Seed points, denoted by $M_0$, (b) $M_0\cup M_1=\{0,1,1+i,-i\}$, (c) $M_0\cup M_1\cup M_2=\{0,1, 1+i,-i, i, 2+i, -1-i, 1-i\}$, (d) $M_0\cup M_1\cup M_2\cup M_3$.}
\label{fig:figure}
\end{figure}

\begin{figure}[!h]
\centering
\includegraphics[scale=1]{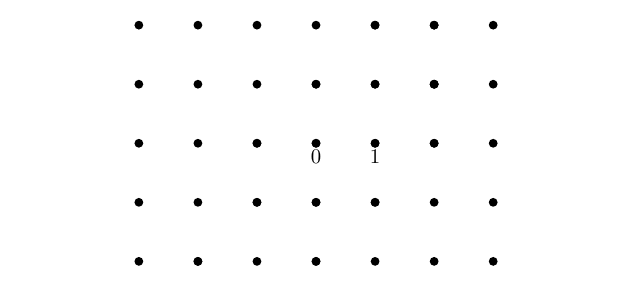}
\caption{The final set $M(U)$, the Gaussian integers.}
\label{gauss}
\end{figure}

%Notice that this is all of the integral linear combinations of $1$ and $i$, the imaginary unit. In fact, this is the construction that gives rise to the Gaussian integers, a known subring of the complex numbers. However, $\left\{0, \frac{\pi}{4}, \frac{\pi}{2}\right\}$ does not form a group, because $\frac{\pi}{4} +  \frac{\pi}{2} = \frac{3\pi}{4}$, and $\frac{3\pi}{4}$ is not an angle used in the construction, so the set of angles is not closed under addition. So this origami ring, $\mathbb{Z}\left[i\right]$, is subject to some constraint other than the angle set that gives $\mathbb{Z}\left[i\right]$ its ring structure. We will work towards identifying it. Also notice that $\mathbb{Z}\left[i\right]$ has \emph{topological} structure, forming a lattice rather than a dense set.\\

\end{example}

As we already mentioned, there are many natural questions about the topological and algebraic structure $M(U)$ as a subset of $\C$. For instance, Example \ref{Gauss} yields a set that exhibits ``nice" structure, by virtue of being both a lattice and a subring of $\C$. In general, we say that $M(U)$ is an \emph{origami ring} if it is a subring of $\C$. In \cite{BBDLG}, Buhler et al. show that when $U$ is a subgroup of $S^1/\{\pm1\}$, $M(U)$ is a ring. In \cite{Moller}, M\"{o}ller generalizes this by finding necessary and sufficient conditions for $M(U)$ to be a ring. The generalization of these algebraic results to higher dimensions remains unstudied.  

In this paper, we focus on the topological structure of the origami construction in higher dimensions. In particular, we ask whether, given a generalization of the construction to higher dimension, we get a lattice or a dense set. 

In the two-dimensional version, these questions have been answered completely. For example, in \cite{BBDLG} it is shown that if $U$ has at least four elements, the set $M(U)$ will be dense in $\C$. The case of only two directions is trivial, but when $U$ has exactly three angles we get the following interesting result obtained independently in both \cite{BR} and \cite{Nedrenco}:

\begin{theorem}[Bahr--Roth, Nedrenco]\label{nedrenco}
If $U$ has exactly three angles, that is $U=\{1,\alpha,\beta\}$,  then $M(U)=\Z+\tau\Z$, where $\tau=[\![0,1]\!]_{\alpha,\beta}$. In other words, $M(U)$ is a lattice. 
\end{theorem}

The converse of this theorem is whether any given lattice can be obtained via an origami construction. In \cite{KS}, the authors show that any ring of integers of an imaginary quadratic field is an origami ring. If we also require that the lattice contain 1, we obtain the following.  

\begin{theorem}
Let $\Lambda$ be a $\Z$-lattice in $\C$ containing 1. Then $\Lambda$ is an origami set.

\end{theorem} 

\begin{proof}
Since we know $\Lambda$ contains 1, we can write it as $\Lambda=\Z+\omega\Z$. Let $$\alpha=\dfrac{\omega}{\|\omega\|}\hspace{1cm} \text{ and } \hspace{1cm}\beta=\dfrac{\omega-1}{\|\omega-1\|}.$$ Then applying the intersection formula \ref{equation1}, we see that $$[\![0,1]\!]_{\alpha,\beta}=\omega.$$ Since we have found angles $\alpha$ and $\beta$ such that $\omega=[\![0,1]\!]_{\alpha,\beta}$, by Theorem \ref{nedrenco} we know $\Lambda=M(\{1,\alpha,\beta\})$. 

\end{proof}

For our $n$-dimensional generalization, we first extend the definition of the intersection operator to $\R^n$. Let $1$ denote the vector $(1,0,0,\dots,0)$ and let $\alpha$ and $\beta$ be two directions represented by points on the unit sphere $S^n$ where we consider antipodal points as equivalent. Let $p$ and $q$ be points in $\R^n$. Then, just as in Equation \ref{intersection}, we define the intersection of the line through $p$ in the direction of $\alpha$ and the line through $q$ in the direction of $\beta,$ denoted by $$[\![p,q]\!]_{\alpha, \beta},$$ as the unique point
\begin{equation*}
z=p+r\alpha=q+s\beta,
\end{equation*} if it exists. If no such point exists, then we write $[\![p,q]\!]_{\alpha, \beta} \colonequals \emptyset.$

Then $M(U)$ is defined in exactly the same way as Definition \ref{origami}, where the only difference is that not all lines are going to intersect, and that we will be constructing a subset of $\R^n$. In the main result of this paper, we give sufficient conditions for the construction to obtain a full lattice. By the latter we mean an $n$-dimensional lattice in $\R^n$. 

\begin{theorem}[Main Theorem] 
\label{main}
Any full $\Z$-lattice in $\R^n$ of the form $\Z+ \Z \tau_1+ \cdots \Z \tau_n$ can be written as $M(U)$ for some angle set $U=\{\alpha_1, \dots, \alpha_m\}$ of $m$ distinct angles including the angle 1, where the angles are represented by elements of $\R^n$ up to scaling, and $m \geq 2n-1$.
Moreover, if $M(U)$ is a full lattice in $\R^n$, then $U$ contains a subset of the form $\{1, \alpha_1, \alpha_1', \dots, \alpha_{n-1}, \alpha_{n-1}'\}$, where $\{1, \alpha_1, \alpha_2, \dots, \alpha_{n-1}\}$ are linearly independent. %Let $V$ be an $\R$-vector space of dimension $n$. Then, any full $\Z$-lattice in $V$ of the form $\Z+ \Z \tau_1+ \cdots \Z \tau_n$ can be written as $M(U)$ for some angle set $U=\{1, \alpha_1, \dots, \alpha_m\}$ of distinct angles containing 1, where the angles are represented by elements of $V$ up to scaling. Moreover, if $M(U)$ is a full lattice in $V$, then $m \geq 2(n-1)$, then $U$ contains a subset of the form $\{1, \alpha_1, \alpha_1', \dots, \alpha_n, \alpha_n'\}$, where $\{1, \alpha_1, \alpha_2, \dots, \alpha_n\}$ is linearly independent.

%Let $B$ be an $n$-dimensional \R$-vector space. Then, any full $\Z$-lattice in $B$ containing 1 can be written as $M(U)$ for some set $U=\{1, \alpha_1, \dots, \alpha_{2(n-1)}\}$ of distinct angles containing 1, where the angles are represented by elements of $B$. Moreover, if $M(U)$ is a full lattice in $B$, no three of the $\alpha_i\in U$ lie in the same 2-dimensional subspace, and for each $\alpha_i \in U$, there is an $\alpha_j \in U$ such that $\{0,1,\alpha_i, \alpha_j\}$ all lie in the same 2-dimensional subspace of $B$.\as{Check this wording again -- in particular the necessary and sufficient conditions established in sections 3 and 4.}
\end{theorem}

Because we are adding many more degrees of freedom to our space, the situation is more subtle, and there is much that remains to be explored. For example, the necessary conditions for obtaining a lattice are not so clear, and we study the contrapositive; that is, sufficient conditions to obtain a dense set, in a few situations.

In Section 2, we prove the first part of our main theorem, which states that any $n$-dimensional lattice is an origami set. In Section 3, we explore the very interesting example of quaternion algebras, where we can make use of the algebraic structure to define an intersection as in Formula \ref{equation1}. Section 4 proves the second part of our theorem, which gives a necessary structure for an angle set that yields a lattice. Section 5 concludes  by exploring some further results regarding density.

\section{Sufficient conditions for a full lattice}

In this section, we give sufficient conditions to obtain a full lattice in $\R^n$. We first present the following lemma, which extends some properties of the intersection operator to $n$ dimensions, and which will streamline our later computations. 

\begin{lemma}\label{lemma3.1}
Let $p,q ,\alpha, \beta \in \R^n$, where $\alpha, \beta$ represent directions. Then,

\begin{enumerate}
\item $[\![p,q]\!]_{\alpha, \beta}=[\![q,p]\!]_{\beta, \alpha}$.
\item $[\![p,q]\!]_{\alpha, \beta}+[\![q,p]\!]_{\alpha, \beta}=p+q$.
\item $[\![a+p,a+q]\!]_{\alpha, \beta}=a+[\![p,q]\!]_{\alpha, \beta}$ for any $a \in \R^n$. \label{lemma3.1.3}
\item $[\![kp,kq]\!]_{\alpha, \beta}=k[\![p,q]\!]_{\alpha, \beta}$ for any $k \in \R$.
\end{enumerate}

\end{lemma}

\begin{proof}
(1): This follows directly from Equation \ref{intersection}.

%\db{rewrite without formula} \s{Done}

(2): Suppose $p+r\alpha=q+s\beta=[\![p,q]\!]_{\alpha, \beta}$. Then, $p-s\beta=q-r\alpha=[\![q,p]\!]_{\alpha,\beta}$, so $[\![p,q]\!]_{\alpha, \beta}+[\![q,p]\!]_{\alpha, \beta}=(p+r\alpha)+(q-r\alpha)=p+q$.

(3): Suppose $p+r\alpha=q+s\beta=[\![p,q]\!]_{\alpha, \beta}$. Then, $a+[\![p,q]\!]_{\alpha, \beta}=a+p+r\alpha=a+q+s\beta=[\![a+p,a+q]\!]_{\alpha,\beta}$.

(4): Suppose $p+r\alpha=q+s\beta=[\![p,q]\!]_{\alpha, \beta}$. Then, $k[\![p,q]\!]_{\alpha,\beta}=k(p+r\alpha)=k(q+s\beta)$. Therefore, $kp+kr\alpha=kp+ks\beta=[\![kp,kq]\!]_{\alpha,\beta}$.
\end{proof}

We are now ready to show that any full lattice containing 1 in an $n$-dimensional $\R$-vector space can be obtained from an origami construction using an angle set of size $2n-1$.

%\begin{proposition} \label{prop1}
%Let $U=\{1, \alpha_1, \dots, \alpha_{2n-1}\}$ be a set of directions represented by elements of an $\R$-vector space $V$ up to scaling. Suppose also that the set $S=\{[\![1,0]\!]_{\alpha_i, \alpha_j}\} \cup \{1\}$ contains exactly $n$ $\R$-linearly independent elements $1, \tau_1, \tau_2, \dots, \tau_{n-1}$ in $V$. Then, the set $M(U)$ is a full lattice in $V$, generated by $S$. Moreover, the set $U$ is of the form $\{1, \alpha_1, \alpha_1', \alpha_2, \alpha_2', \dots, \alpha_{n-1}, \alpha_{n-1}'\}$ where $\alpha_i'=\alpha_i-1$ and no three of the $\alpha_i$ are linearly dependent.
%\end{proposition}

\begin{theorem} \label{thm1}
Let $U=\{1, \alpha_1, \alpha_1', \alpha_2, \alpha_2', \dots, \alpha_{n-1}, \alpha_{n-1}'\}$ be a set of directions represented by elements of $\R^n$ up to scaling, where $\alpha_i'=\dfrac{\alpha_i-1}{\|\alpha_i-1\|}$ and the set $U'\colonequals \{1, \alpha_1, \alpha_2, \dots, \alpha_{n-1}\}$  is linearly independent. Then, the set $M(U)=\Z+ \Z \tau_1+ \Z\tau_2 + \cdots + \Z \tau_{n-1}$ is a full lattice in $\R^n$, where $\tau_i\colonequals [\![0,1]\!]_{\alpha_i,\alpha_i '}$.  \end{theorem}

%\as{Fix the proof to match the tau notation.}
%\s{Fixed the proof--one issue arose: when we require the angles to be unit length, $\alpha_i-1$ is a different angle than it would be if we scale $\alpha_i$. This is ok, the proof still works, but all of the $\tau_i$'s just end up being $\alpha_i$ which is unit length. We could alternatively allow for $\alpha_i'$ to be $\alpha_i-t_i$ where $t_i \in \R$ to allow for different values of $\tau_i$. I am not sure it is worth it because it is no longer an issue when we start with the generators $\tau_i$ in the next section and define the angles from $\tau_i$; i.e., $\alpha_i=\frac{\tau_i}{||\tau_i||}$ and $\alpha_i'=\frac{\tau_i-1}{||\tau_i-1||}$ rather than $\frac{\alpha_i-1}{||\alpha_i-1||}$.}

\begin{proof}

%\as{For new version of proposition, we need a statement on why the taus are linearly independent.}

%If the set $S=\{1, \alpha_1, \alpha_2, \dots, \alpha_{n-1}\}$ is linearly independent, then so is the set $\{1, \tau_1, \tau_2, \dots, \tau_{n-1}\}$, for each $\tau_i$ is a scalar multiple of $\alpha_i$ by the definition of the intersection.%First suppose that the set $S=\{[\![1,0]\!]_{\alpha_i, \alpha_j}\} \cup \{1\}$ is linearly independent and contains exactly $n$ elements $1, \tau_1, \tau_2, \dots, \tau_{n-1}$ in $V$. Then, $U$ must be of the form $$\{1, \alpha_1, \alpha_1', \alpha_2, \alpha_2', \dots, \alpha_{n-1}, \alpha_{n-1}'\}$$ with $\tau_i=[\![0,1]\!]_{\alpha_i, \alpha_i'}$ for each $i \in \{1,2, \dots, n-1\}$ for if any one angle is used to obtain two of the $\tau_i$, then the set $S$ would not be linearly independent. Note here that $\tau_i=[\![0,1]\!]_{\alpha_i, \alpha_i'}=r_i\alpha_i=1+s_i\alpha_i'$ for $r_i,s_i \in \R^\times$, so $\alpha_i=\frac{\tau_i}{r_i}$ and $\alpha_i'=\frac{\tau_i-1}{s_i}$. Then, since $S$ is linearly independent, so are any set of three from the set $U\setminus \{1\}=\{\frac{\tau_1}{r_1}, \frac{\tau_1-1}{s_1}, \frac{\tau_2}{r_2}, \frac{\tau_2-1}{s_2}, \dots, \frac{\tau_{n-1}}{r_{n-1}}, \frac{\tau_{n-1}-1}{s_{n-1}}$\}.

Let $\Lambda \colonequals\Z+\Z \tau_1+ \Z \tau_2 + \cdots + \Z \tau_{n-1}$. We will show that $\Lambda=M(U)$, so we first show that $\Lambda \subseteq M(U)$. Define $\Lambda_k \colonequals \Z+\Z \tau_1 + \cdots + \Z \tau_{k}$. We proceed by induction on $k$ to show that $\Lambda_k \subseteq M(U)$ for all $k \leq n-1$. We know by Nedrenco \cite{Nedrenco} that $M(\{1, \alpha_1, \alpha_1'\})=\Z+\Z \tau_1$ because $[\![0,1]\!]_{\alpha_1, \alpha_1'}=\tau_1$. Therefore $\Z+\Z \tau_1 =\Lambda_1 \subseteq M(U)$ and the base case is satisfied. Now, suppose that $\Z+\Z \tau_1+ \cdots + \Z \tau_k \subseteq \Lambda$ for some $k \in \{1,2,, \dots{n-2}\}$. Then, for any integer $m$, and element $\ell \in \Lambda_k$, we need to show that $m\alpha_{k+1}+ \ell \in M(U)$.  Now, $\tau_{k+1}=[\![0,1]\!]_{\alpha_{k+1}, \alpha_{k+1}'}$, so $m\tau_{k+1}=[\![0,m]\!]_{\alpha_{k+1}, \alpha_{k+1}'}$ by Lemma \ref{lemma3.1} (4). Then, by Lemma \ref{lemma3.1} (3), $[\![\ell, m+\ell]\!]_{\alpha_{k+1}, \alpha_{k+1}'}=m\tau_{k+1}+\ell$ which is in $M(U)$ since $m+\ell$ and $\ell$ are in $M(U)$. Therefore $\Lambda_k \subseteq M(U)$ for all $k \leq n-1$, so in particular $\Lambda \subseteq M(U)$.

We now show that $M(U) \subseteq \Lambda$. Both seed points 0 and 1 are contained in $\Lambda$. It therefore suffices to show that for any $p,q \in \Lambda$ and $\xi, \varphi \in U$, we also have $[\![p,q]\!]_{\xi, \varphi} \in \Lambda$. By Lemma \ref{lemma3.1} (3), we may shift by $-p$ to suppose $p=0$ and $q=a_0+a_1\tau_1+a_2\tau_2+ \cdots + a_{n-1}\tau_{n-1} \in \Lambda$. Values of $[\![0,q]\!]_{\xi, \varphi}$ can be found in Table \ref{table}, where $[\![0,q]\!]_{\xi, \varphi}=r\xi=q+s\varphi$ for $r,s \in \R$ whose values are also listed.

\begin{center}
\begin{table}[!h]
\begin{tabular}{|c|c|c|c|c|} \hline
\textbf{Angles}& \textbf{Intersection} & $r$ & $s$ & \textbf{Conditions}\\ \hline
%$1, \alpha_1$ & $[\![0,q]\!]_{1, \alpha_1}=a_0$ & $a_0$ & $-a_1$ & $a_i=0$ for $i\neq 0,1$\\ \hline
%$1, \alpha_1'$ & $[\![0,q]\!]_{1, \alpha_1'}=a_0+a_1$ & $a_0+a_1$ & $-a_1$ & $a_i=0$ for $i\neq 0,1$\\ \hline
%$1, \alpha_2$ & $[\![0,q]\!]_{1, \alpha_2}=a_0$ & $a_0$ & $-a_2$ & $a_i=0$ for $i \neq 0,2$\\ \hline
%$1, \alpha_2'$ & $[\![0,q]\!]_{1, \alpha_2'}=a_0+a_2$ & $a_0+a_2$ & $-a_2$ &  $a_i=0$ for $i \neq 0,2$\\ \hline
%$\alpha_1, \alpha_1'$ & $[\![0,q]\!]_{\alpha_1, \alpha_1'}=(a_0+a_1)\alpha_1$ & $a_0+a_1$ & $a_0$ &$a_i=0$ for $i \geq 2$\\ \hline
%$\alpha_1, \alpha_2$ & $[\![0,q]\!]_{\alpha_1, \alpha_2}=a_1\alpha_1$ & $a_1$ & $-a_2$ &$a_0=0$ for $i \neq 1,2$\\ \hline
%$\alpha_1, \alpha_2'$ & $[\![0,q]\!]_{\alpha_1, \alpha_2'}=a_1\alpha_1$ & $a_1$ & $-a_0=-a_2$ & $a_0=a_2$\\ \hline
%$\alpha_1', \alpha_2$ & $[\![0,q]\!]_{\alpha_1', \alpha_2}=-a_0(\alpha_1-1)$
%& $-a=a_1$ & $-a_2$ & $a_0=-a_1$\\ \hline
%$\alpha_1', \alpha_2'$ & $[\![0,q]\!]_{\alpha_1', \alpha_2'}=a_1(\alpha_1-1)$ & $a_1$& $-a_2$ & $a+a_1=-a_2$\\ \hline
%$\alpha_2, \alpha_2'$ & $[\![0,q]\!]_{\alpha_2, \alpha_2'}=(a+a_2)\alpha_2$ & $a_0+a_2$ & $a_0$ &$a_1=0$\\ \hline

$1, \alpha_i$ & $[\![0,q]\!]_{1, \alpha_i}=a_0$ & $a_0$ & $-a_i$ & $a_k=0$ for $k\neq 0,i$\\ \hline
$1, \alpha_i'$ & $[\![0,q]\!]_{1, \alpha_i'}=a_0+a_i$ & $a_0+a_i$ & $-a_1$ & $a_k=0$ for $k\neq 0,i$\\ \hline
$\alpha_i, \alpha_i'$ & $[\![0,q]\!]_{\alpha_i, \alpha_i'}=(a_0+a_i)\tau_i$ & $a_0+a_i$ & $a_0$ &$a_k=0$ for $k \neq 0,i$\\ \hline
$\alpha_i, \alpha_j$ & $[\![0,q]\!]_{\alpha_i, \alpha_j}=a_i\tau_i$ & $a_i$ & $-a_j$ &$a_k=0$ for $k \neq i,j$\\ \hline
$\alpha_i, \alpha_j'$ & $[\![0,q]\!]_{\alpha_i, \alpha_j'}=a_i\tau_i$ & $a_i$ & $a_0$ & $a_0=-a_j$; $a_k=0$ for $k\neq 0,i,j$\\ \hline
$\alpha_i', \alpha_j$ & $[\![0,q]\!]_{\alpha_i', \alpha_j}=a_i(\tau_i-1)$& $a_i$ & $-a_j$ & $a_0=-a_i$; $a_k=0$ for $k\neq 0,i,j$\\ \hline
$\alpha_i', \alpha_j'$ & $[\![0,q]\!]_{\alpha_i', \alpha_j'}=a_i(\tau_i-1)$ & $a_i$& $-a_j$ & $a_0+a_i=-a_j$; $a_k=0$ for $k\neq 0,i,j$\\ \hline

\end{tabular}
\caption{The lattice is closed under intersections.}
\label{table}
%\s{If someone could check these table entries, that would be great!}
\end{table}
\end{center}

%\begin{center}
%\begin{tabular}{|c|c|c|c|c|} \hline
%\textbf{Angles}& \textbf{Intersection} & $r$ & $s$ & \textbf{Restrictions}\\ \hline
%$1, \alpha_1$ & $[\![p,0]\!]_{1, \alpha_1}=a_1(\tau_1-1)$ & $-a-a_1$ & $a_1$ & Need $a_2=a_3= \dots=0$\\ \hline
%$1, \alpha_1'$ & $[\![p,0]\!]_{1, \alpha_1'}=a_1\tau_1$ & $-a$ & $a_1$ & Need $a_2=a_3=\dots =0$\\ \hline
%$1, \alpha_2$ & $[\![p,0]\!]_{1, \alpha_2}=a_1(\tau_2-1)$ & $-a-a_2$ & $a_2$ & Need $a_1=a_3=a_4=\dots=0$\\ \hline
%$1, \alpha_2'$ & $[\![p,0]\!]_{1, \alpha_2'}=a_1\tau_2$ & $-a$ & $a_3$ & Need $a_2=a_3=\cdots=0$\\ \hline
%$\alpha_1, \alpha_1'$ & $[\![p,0]\!]_{\alpha_1, \alpha_1'}=(a+a_1)\tau_1$ & $a$ & $a_1$ &Need $a_2, a_3, \dots=0$\\ \hline
%$\alpha_1, \alpha_2$ & $[\![p,0]\!]_{\alpha_1, \beta_1}=a_2(\omega-1)$ & $-a_1$ & $a_2$ & Need $a+a_1+a_2=0; d=0$\\ \hline
%$\alpha_1, \alpha_2'$ & $[\![p,0]\!]_{\alpha_1, \alpha_2'}=a_2(\tau_2)$ & $a=-a_1$ & $a_2$ & Need $a=-a_1; a_3, a_4, \dots=0$\\ \hline
%$\alpha_1', \alpha_2$ & $[\![p,0]\!]_{\alpha_1', \alpha_2}=a_2(\tau_2-1)=-a(\tau_2-1)$ & $-a_1$ & $-a=a_2$ & Need $a=-a_2; a_3=a_4=\dots=0$\\ \hline
%$\alpha_1', \alpha_2'$ & $[\![p,0]\!]_{\alpha_1', \alpha_2'}=a_2(\tau_2)$ & $-a_1$& $a_2$ & Need $a=a_3=a_4=\dots=0$\\ \hline
%$\alpha_2, \alpha_2'$ & $(a+a_2)\tau_2$ & $a$ & $a+a_2$ & Need $a_1=a_3=a_4=\dots=0$ \\ \hline
%\end{tabular}
%\end{center}

All intersections listed are in $\Lambda$ since each $a_i$ is an integer. Therefore, for any angles $\xi, \varphi \in U$ and $q \in \Lambda$, we have $[\![0,q]\!]_{\xi, \varphi} \in \Lambda$ so $M(U) \subseteq \Lambda$. We then have $M(U)=\Lambda$.

\end{proof}

We may also obtain a given element $\tau \in V$, as $[\![0,1]\!]_{\frac{\tau}{\|\tau\|}, \frac{\tau-1}{\|\tau-1\|}}=\tau$ by Formula \ref{intersection} since $$0+\|\tau\|\left(\frac{\tau}{\|\tau\|}\right)=\tau=1+\|\tau-1\|\left(\frac{\tau-1}{\|\tau-1\|}\right).$$ This leads to the following corollary.

\begin{corollary} \label{cor1} Let $\Lambda=\Z+\Z \tau_1 +\cdots  + \Z \tau_{n-1}$ be a full lattice in $\R^n$. Then, letting $U=\{1, \alpha_1, \alpha_1', \dots, \alpha_{n-1}, \alpha_{n-1}'\}$, where $\alpha_i=\frac{\tau_i}{\|\tau_i\|}$ and $\alpha_i'=\frac{\tau_i-1}{\|\tau_i-1\|}$, we have $M(U)=\Lambda$.
\end{corollary}

\begin{proof}
This follows from the intersection formula and Theorem \ref{thm1}.
\end{proof}

%\s{examples in 3d}

We conclude this section by visualizing origami constructions in 3 dimensions. In this case, we abuse notation slightly by using $1$, $i$, and $j$ to denote the vectors $(1,0,0) $, $(0,1,0)$ and $(0,0,1)$, respectively. We choose this notation to be consistent with complex numbers and the Hamilton quaternion notation, which will be at the core of the next section. 

\begin{example}\label{cube}
Starting with the origami construction of the Gaussian integers where $U=\left\{1, i, \frac{1}{\sqrt{2}}(1+i)\right\},$ we add two angles in a different orthogonal plane, $j$ and $\frac{1}{\sqrt{2}}(1+j).$ This construction gives the lattice $\Z + \Z i + \Z j,$ as illustrated in Figure \ref{3d}.
%\db{Should I change these angles to have norm 1?} Adriana: I just normalized the angles. The picture remains the same, as do the lattice points and intersection points.  

\begin{center}
\begin{figure}[h]
\includegraphics[scale = 1]{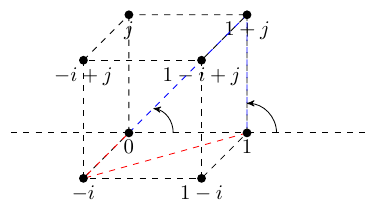}
\caption{Some initial points of the origami construction where $U=\left\{1, i, j, \frac{1}{\sqrt{2}}(1+i), \frac{1}{\sqrt{2}}(1+j)\right\}$.}
\label{3d}
\end{figure}
\end{center}

\end{example}

\section{Examples in quaternion algebras}

%As we have shown, we may construct any lattice in $\R^n$ containing 1. In two dimensions, we constructed the Gaussian integers in Example \ref{Gauss}. A natural extension of this example to four dimensions has to do with so-called quaternion algebras.  this means that in particular, we can construct any quaternion $\Z$-order in the quaternion algebra $\left(\frac{-1, -1}{\Q}\right)$.

Let $B=\R+ \R i + \R j + \R k$ be the Hamilton quaternion algebra, where $i^2=j^2=k^2=ijk=-1$. In particular, $B$ is a 4-dimensional vector space over $\R$, with a multiplication structure. This situation is analogous to the complex numbers, and in particular allows us to compute the intersection operator through a formula similar to that in Equation \ref{equation1}. 

%Then, we have the following formula for the intersection of two lines through points $p$ and $q$ in $B$, respectively.

\begin{proposition} \label{formula}
For any $p$ and $q$ and invertible $\alpha, \beta$, and such that $\alpha\overline{\beta}-\beta\overline{\alpha}\neq 0$ we can compute the intersection with the following formula:
 $$[\![ p,q ]\!]_{\alpha, \beta}=p+r\alpha=p+[\beta(\overline{p-q})-(p-q)\overline{\beta}](\alpha\overline{\beta}-\beta\overline{\alpha})^{-1}\alpha.$$
\end{proposition}

\begin{remark}
Note that the case where $\alpha\overline{\beta}-\beta\overline{\alpha}= 0$ corresponds to the case when the lines have no intersection.  If $\alpha\overline{\beta}=\beta\overline{\alpha}$, then $\alpha\overline{\beta}$ is its own conjugate and is thus in $\R$; i.e., $\alpha\overline{\beta}=r$. Then, $\alpha\overline{\beta}\beta=r\beta$ so $\alpha\|\beta\|=r\beta$. Note that $\|\beta\|$ and $r$ are real numbers and hence $\alpha$ and $\beta$ are scalar multiples of each other and represent the same angle. It should also be noted that there are other ways in order for this intersection to be undefined besides the two angles being the same. For example, when the two lines are skew, the values of $r$ and $s$ in the equation $p+r\alpha=q+s\beta$ are not real numbers.
\end{remark}

\begin{proof}

First, solve for $s$: $p-q+r\alpha=s\beta$ so $s=(p-q+r\alpha)\beta^{-1}=(p-q+r\alpha)\overline{\beta}$ since points representing the angles are chosen to have unit length.

Since $s \in \R$, we must have $s=\overline{s}$. Therefore: $$(p-q+r\alpha)\overline{\beta}=\beta(\overline{p-q}+r\overline{\alpha}).$$

Solving for $r$, we have
\begin{align*}
(p-q)\overline{\beta}+r(\alpha \overline{\beta})&=\beta(\overline{p-q})+r(\beta\overline{\alpha})\\
r(\alpha \overline{\beta})-r(\beta\overline{\alpha})&=\beta(\overline{p-q}-(p-q)\overline{\beta}\\
r(\alpha\overline{\beta}-\beta \overline{\alpha})&=\beta(\overline{p-q})-(p-q)\overline{\beta}\\
r&=[\beta(\overline{p-q})-(p-q)\overline{\beta}](\alpha\overline{\beta}-\beta\overline{\alpha})^{-1}
\end{align*}

Since we require $r \in \R$, we have $r=\overline{r}$. Then,
\begin{align*}
(\beta\overline{\alpha}-\alpha\overline{\beta})^{-1}[(p-q)\overline{\beta}-\beta(\overline{p-q})]&=(\alpha\overline{\beta}-\beta\overline{\alpha})^{-1}[\beta(\overline{p-q})-(p-q)\overline{\beta}]\\
&=\overline{r}\\
&=r\\
&=[\beta(\overline{p-q})-(p-q)\overline{\beta}](\alpha\overline{\beta}-\beta\overline{\alpha})^{-1}.
\end{align*}

In other words, the two terms $(\alpha\overline{\beta}-\beta\overline{\alpha})^{-1}$ and $[\beta(\overline{p-q})-(p-q)\overline{\beta}]$ commute with each other. Both are pure imaginary as well. They therefore must be scalar multiples of each other.
Now, the formula becomes $$[\![ p,q ]\!]_{\alpha, \beta}=p+r\alpha=p+[\beta(\overline{p-q})-(p-q)\overline{\beta}](\alpha\overline{\beta}-\beta\overline{\alpha})^{-1}\alpha.$$
\end{proof}

\begin{example}\label{lipschitz}
Let $B=\R+ \R i + \R j + \R k$ be the Hamilton quaternion algebra, as above. Using Proposition \ref{formula}, and listing $\emptyset$ when the intersection is undefined, Table \ref{table:L} gives the initial intersections $[\![0,1]\!]_{\alpha, \beta}$, where $$\alpha, \beta \in U=\left\{1, i, \frac{i-1}{\sqrt{2}}, j, \frac{j-1}{\sqrt{2}}, k, \frac{k-1}{\sqrt{2}}\right\},$$ with $\alpha$ listed in the columns and $\beta$ in the rows.
\bigskip

\begin{table}[!h]

\begin{tabular}{|c|c|c|c|c|c|c|c|}
\hline
$\alpha \backslash \beta$ & 1 & $i$ & $\frac{i-1}{\sqrt{2}}$ & $j$ & $\frac{j-1}{\sqrt{2}}$ & $k$ & $\frac{k-1}{\sqrt{2}}$\\ \hline
1 & $\emptyset$ & 1& 1& 1& 1& 1 & 1\\ \hline
$i$ & 0 & $\emptyset$ & $i$ & $\emptyset$ & $\emptyset$&$\emptyset$&$\emptyset$\\ \hline
$\frac{i-1}{\sqrt{2}}$ & 0 & $1-i$ & $\emptyset$&$\emptyset$&$\emptyset$&$\emptyset$&$\emptyset$\\ \hline
$j$ & 0 & $\emptyset$&$\emptyset$&$\emptyset$& $j$ & $\emptyset$&$\emptyset$\\ \hline
$\frac{j-1}{\sqrt{2}}$ & 0 & $\emptyset$&$\emptyset$& $1-j$ & $\emptyset$&$\emptyset$&$\emptyset$\\ \hline
$k$ & 0 & $\emptyset$&$\emptyset$&$\emptyset$&$\emptyset$&$\emptyset$& $k$\\ \hline
$\frac{k-1}{\sqrt{2}}$ & 0 & $\emptyset$&$\emptyset$&$\emptyset$&$\emptyset$& $1-k$ & $\emptyset$\\ \hline
\end{tabular}
\caption{The initial intersections for the origami construction of the Lipschitz order.}
\label{table:L}
\end{table}

By Corollary \ref{cor1}, $M(U)=\Z+\Z i+\Z j + \Z k$. This lattice is also known as the Lipschitz order.

\end{example}

The Lipschitz order appears as the generalization of the Gaussian integers to 4 dimensions. However, it is not a maximal $\Z$-order in the Hamilton quaternion algebra for it is contained in the maximal order $\Z+\Z i+\Z j + \Z \frac{1+i+j+k}{2}$, which is known as the Hurwitz order in $B$.

\begin{example}
Using Proposition \ref{formula} again, and listing $\emptyset$ when the intersection is undefined, Table \ref{table:H} gives the initial intersections $[\![0,1]\!]_{\alpha, \beta}$, where $$\alpha, \beta \in U=\left\{1, i, i-1, j, j-1, \frac{1+i+j+k}{2},\frac{1-i-j-k}{2}\right\},$$ with $\alpha$ listed in the columns and $\beta$ in the rows.

\begin{table}[!h]

\begin{tabular}{|c|c|c|c|c|c|c|c|}
\hline
$\alpha \backslash \beta$ & 1 & $i$ & $\frac{i-1}{\sqrt{2}}$ & $j$ & $\frac{j-1}{\sqrt{2}}$ & $\frac{1+i+j+k}{2}$ & $\frac{1-i-j-k}{2}$\\ \hline
1 & $\emptyset$ & 1& 1& 1& 1& 1 & 1\\ \hline
$i$ & 0 & $\emptyset$ & $i$ & $\emptyset$ & $\emptyset$&$\emptyset$&$\emptyset$\\ \hline
$\frac{i-1}{\sqrt{2}}$ & 0 & $1-i$ & $\emptyset$&$\emptyset$&$\emptyset$&$\emptyset$&$\emptyset$\\ \hline
$j$ & 0 & $\emptyset$&$\emptyset$&$\emptyset$& $j$ & $\emptyset$&$\emptyset$\\ \hline
$\frac{j-1}{\sqrt{2}}$ & 0 & $\emptyset$&$\emptyset$& $1-j$ & $\emptyset$&$\emptyset$&$\emptyset$\\ \hline
$\frac{1+i+j+k}{2}$ & 0 & $\emptyset$&$\emptyset$&$\emptyset$&$\emptyset$&$\emptyset$& $\frac{1+i+j+k}{2}$\\ \hline
$\frac{1-i-j-k}{2}$ & 0 & $\emptyset$&$\emptyset$&$\emptyset$&$\emptyset$& $\frac{1-i-j-k}{2}$ & $\emptyset$\\ \hline
\end{tabular}
\caption{The initial intersections for the origami construction of the Hurwitz order.}
\label{table:H}
\end{table}

By Corollary \ref{cor1}, $M(U)=\Z+\Z i+\Z j + \Z \frac{1+i+j+k}{2}$. \end{example}

%\begin{example}
%Let $B=\R+ \R i + \R j + \R k$ again be the quaternion algebra, where $i^2=j^2=k^2=ijk=-1$. Using this formula, listing $\emptyset$ when the intersection is undefined, the following table gives the intersections for seed points 0 and 1, and angle set $U=\{1, i, \frac{i+1}{\sqrt{2}}, \frac{j+1}{\sqrt{2}}, \frac{j-1}{\sqrt{2}},\frac{i+k}{\sqrt{2}}, \frac{i+k-2}{\sqrt{6}}\}$, with $\alpha$ in the columns and $\beta$ in the rows.
%\bigskip
%
%\begin{tabular}{|c|c|c|c|c|c|c|c|}
%\hline
%$\alpha \backslash \beta$ & 1 & $i$ & $\frac{i-1}{\sqrt{2}}$ & $\frac{j+1}{2}$ & $\frac{j-1}{2}$ & $\frac{i+k}{\sqrt{2}}$ & $\frac{i+k-2}{\sqrt{6}}$\\ \hline
%1 & $\emptyset$ & 1& 1& 1& 1& 1 & 1\\ \hline
%$i$ & 0 & $\emptyset$ & $i$ & $\emptyset$ & $\emptyset$&$\emptyset$&$\emptyset$\\ \hline
%$\frac{i-1}{\sqrt{2}}$ & 0 & $1-i$ & $\emptyset$&$\emptyset$&$\emptyset$&$\emptyset$&$\emptyset$\\ \hline
%$\frac{j+1}{\sqrt{2}}$ & 0 & $\emptyset$&$\emptyset$&$\emptyset$& $\frac{1+j}{2}$ & $\emptyset$&$\emptyset$\\ \hline
%$\frac{j-1}{\sqrt{2}}$ & 0 & $\emptyset$&$\emptyset$& $\frac{1-j}{2}$ & $\emptyset$&$\emptyset$&$\emptyset$\\ \hline
%$\frac{i+k}{\sqrt{2}}$ & 0 & $\emptyset$&$\emptyset$&$\emptyset$&$\emptyset$&$\emptyset$& $\frac{i+k}{2}$\\ \hline
%$\frac{i+k-2}{\sqrt{6}}$ & 0 & $\emptyset$&$\emptyset$&$\emptyset$&$\emptyset$& $\frac{2-i-k}{2}$ & $\emptyset$\\ \hline
%\end{tabular}
%
%\bigskip
%
%Therefore, by Corollary \ref{cor1}, $M(U)=\Z+\Z i+\Z \frac{1+j}{2} + \Z \frac{i+k}{2}$.
%\end{example}

The calculations used to compute the table can be done by hand using the formula in Proposition \ref{formula}, or by using software such as Sage \cite{sage}.

\section{Lower bound for $|U|$ in the case where $M(U)$ is a lattice}

Now that we have shown how to obtain a given lattice in $\R^n$, we proceed to investigate the conditions on the angle set that are necessary in order to obtain a lattice. In particular, we prove the first part of Theorem \ref{main}.  

%We first give a necessary condition in order to obtain a lattice to demonstrate why the angle set must have the particular structure shown in Theorem \ref{thm1}.
%\s{subsections}

\begin{proposition} \label{prop2}
Let $U$ be a set of angles in $\R^n$ containing 1. If $M(U)=\Lambda$ is a full lattice in $\R^n$, then $U$ contains at least $2n-1$ angles. Moreover, the set $U$ contains a subset of the form  $U'=\{1, \alpha_1, \alpha_1', \alpha_2, \alpha_2', \dots, \alpha_{n-1}, \alpha_{n-1}'\}$, where $r\alpha_i-s\alpha_i' \in \R+\R\alpha_1+\cdots+\R \alpha_{i-1}$ for some $r,s \in\R$. 

%In particular, $U$ contains at least $2n-1$ angles.%the set $\{1, \alpha_1, \alpha_2, \dots, \alpha_n\}$ is linearly independent. 

%\s{Define $\textbf{1}$ at the beginning}

\end{proposition}
\begin{proof}

Suppose that $U$  is a set of angles in $\R^n$ containing 1 and that $M(U)=\Lambda$ is a full lattice in $\R^n$. Then, in order for $M(U)$ to be non-trivial, there must exist angles $\alpha_1, \alpha_1' \in U$ such that $[\![0,1]\!]_{\alpha_1, \alpha_1'}$ is defined. Define the following:
\begin{align*}
\xi_1 &\colonequals [\![0,1]\!]_{\alpha_1, \alpha_1'}\\
\Lambda_1 &\colonequals \Z+\Z \xi_1\\
V_1 &\colonequals \R+\R\xi_1\\
U_1 &\colonequals U \cap V_1.
\end{align*}

Then, in particular, $\Lambda_1 \subseteq M(U_1) \subseteq V_1$ (in fact, $U_1=\{1, \alpha_1, \alpha_1'\}$, otherwise $M(U_1)$ and hence $M(U)$ would be dense).

We then define $\xi_i, \Lambda_i, V_i$, and $U_i$ iteratively. For each $i$ with $1<i\leq n-1$, there must exist $\alpha_i, \alpha_i' \in U \setminus U_{i-1}$ and an $\ell_{i-1} \in \Lambda_{i-1}$ such that $[\![0,\ell_{i-1}]\!]_{\alpha_i, \alpha_i'} \notin V_{i-1}$, otherwise we would have $M(U) \subseteq V_{i-1}$. When possible, choose $\ell_{i-1}=1$. Then, define the following:
\begin{align*}
\xi_i &\colonequals [\![0,\ell_{i-1}]\!]_{\alpha_i, \alpha_i'}\\
V_i &\colonequals V_{i-1} +\R \xi_i\\
U_i &\colonequals U \cap V_i\\
\Lambda_i &\colonequals M(U_i).
\end{align*}

 Then, $M(U_i) \subseteq V_i$ and $U=\bigcup_{i=1}^{n-1} U_i$. In particular, $$\{1, \alpha_1, \alpha_1', \dots, \alpha_{n-1}, \alpha_{n-1}'\} \subseteq U.$$ Finally, note that $$[\![0,\ell_{i-1}]\!]_{\alpha_i, \alpha_i'}=r\alpha_i=\ell_{i-1}+s\alpha_i$$ so $$r\alpha_i-s\alpha_i'=\ell_{i-1} \in \Lambda_{i-1} \subseteq V_i=\R \xi_1+\cdots+\R \xi_i=\R+ \R \alpha_1+ \cdots + \R \alpha_i.$$
 
 \end{proof}
%\subsection{Origami basis}

One consequence of the result described in Proposition \ref{prop2} is that we obtain a sense of minimality of the set of angles. But it is not clear whether we can always attain the bound. In other words, if $M(U)$ is a lattice, can we always find a subset $U' \subseteq U$ such that $M(U')=M(U)$ and $U'$ has exactly $2n-1$ elements?

%If this were true, then we could define a basis much like when talking about a vector space, where the basis is the generating set of minimal size. In this instance, one could say that an ``origami basis" is any set of angles $U$ of smallest size such that $M(U)$ is an origami set. Moreover, we could define the origami dimension to be $$\odim(M(U))\colonequals\min|U|.$$

%\begin{definition}
%An \emph{origami basis} is any set of angles $U$ of smallest size such that $M(U)$ is an origami set. Moreover, we define the \emph{origami dimension} to be $$\odim(M(U))\colonequals\min|U|.$$
%\end{definition}

%Notice that, in particular, this set won't be unique. One can rephrase the statement in Proposition \ref{prop2} as saying that if $\Lambda \colonequals M(U)$ is an $n$-dimensional lattice in $\R^n$, then $\odim(M(U))=2n-1$. 

%But none of this guarantees the existence of such a set, so one could explore the question: if $M(U)$ is a lattice $\Lambda$, does it contain an origami basis for $\Lambda$? 

It is certainly possible, as we show in Example \ref{yesbasis} below, but not guaranteed, as shown in Example \ref{nobasis}. 

\begin{example} \label{yesbasis}
If $$U\colonequals \left\{1,i, \frac{i-1}{\sqrt{2}}, j, \frac{j-1}{\sqrt{2}}, k, \frac{k-1}{\sqrt{2}}, \frac{1+i+j+k}{2}, \frac{-1+i+j+k}{2}\right\},$$ then we can choose $$U'\colonequals \left\{1,i, \frac{i-1}{\sqrt{2}}, j, \frac{j-1}{\sqrt{2}}, \frac{1+i+j+k}{2}, \frac{-1+i+j+k}{2}\right\}$$ and we have $$M(U)=M(U')=\Z+\Z i+\Z j + \Z \frac{1+i+j+k}{2}.$$
\end{example}

%From experimenting with concrete examples such as the one above, we suspect that it is always  possible to find an origami basis. We believe that this would be easier to prove if there were clear criteria for when one is guaranteed a lattice versus a dense set. This is the content of our last section. 

\begin{example} \label{nobasis}
If $$U=\left\{ 1, i, \frac{i-1}{\sqrt{2}}, j, \frac{1+i+j}{\sqrt{3}}, \frac{2+3i+7j}{\sqrt{62}}, \frac{-9-8i+7j}{\sqrt{194}}\right\},$$ then $$M(U)=\Z+ \Z i+ \Z \frac{(5+2i+j)}{11},$$ which can also be written as $$M(U)=\Z+\Z i+\Z j+\Z\frac{(2+3i+7j)}{11}.$$
For any $U' \subsetneq U,$ it is straightforward, but tedious, to show that $M(U') \subsetneq M(U).$
\end{example}

\section{Introducing new angles}

We finish this paper by exploring how one can alter the angle set of a lattice to obtain a dense origami set. If $M(U)$ is a lattice, there are three scenarios that may occur when new angles are introduced to the angle set $U$:

\begin{enumerate}

\item a finer lattice is obtained,
\item  the set becomes dense, or
\item no change occurs and the same lattice is obtained.

\end{enumerate}

The conditions under which each of these three outcomes occur are not exhaustive; there are several ways to achieve each of the three outcomes. In fact, altering the angle set in the slightest way (e.g., by adding, removing, or changing a single angle) may change the outcome.
\subsection{Obtaining a finer lattice} 
In Proposition \ref{prop2}, from a given angle set $U$ with $M(U)$ a lattice, we found a subset $U'$ containing $2n-1$ angles, and in Example \ref{nobasis} we obtained that $M(U) \supsetneq M(U')$ is actually a finer lattice. It is therefore to be expected that, in some instances, one can add an angle to the angle set for a given lattice and obtain a finer lattice.

\begin{example}\label{finerlattice}

Consider $\R^3$ and the angle sets $$U\colonequals \left\{1,i, \frac{i-1}{\sqrt{2}}, j, \frac{1+i+j}{\sqrt{3}}\right\}$$ and $$\widetilde{U}\colonequals \left\{1,i, \frac{i-1}{\sqrt{2}}, j, \frac{1+i+j}{\sqrt{3}}, \frac{1-2i+j}{\sqrt{6}}\right\} \supset U,$$ Then, $M(U)=\Z+ \Z i + \Z j ,$ and $M(\widetilde{U})=\Z+\Z i+\Z \frac{1+i+j}{3},$ and so $M(\widetilde{U})$ is indeed a finer lattice. 

%Note that equivalently $\mathcal{O}= \Z+ \Z i + \Z k + \Z \frac{1+i+j+k}{2}=\Z+ \Z j + \Z k + \Z \frac{1+i+j+k}{2}$. Now, for any elements $p$ and $q$ in the Hurwitz order, we have $[\![p,q]\!]_{\alpha,\beta}\in \mathcal{O}$ because the Hurwitz order can be constructed using any of the angles from $U$. Therefore, for any angles $\alpha,\beta \in U$, we also have that $M(U)=\mathcal{O}$ which is in fact a finer lattice than $M(U')$.

\end{example}

\subsection{Obtaining a dense set}
There are a few ways for an origami set in $\R^n$ to be dense. One is for the lines through the origin at four different angles in $U$ to lie in the same plane, as is the case in two dimensions \cite{BBDLG}. %\s{wording?} 

\begin{theorem} If $U$ is a set of angles represented by elements of $\R^n$ whose $\R$-span is $\R^n$ and at least four angles exist in the same two-dimensional subspace, then $M(U)$ is dense in $\R^n.$ 
\end{theorem}
\begin{proof}
Consider the origami set made from angles in $U.$ From the work of Buhler et al., we know that in a two-dimensional subspace, an origami set is dense if the set of angles contains more than three angles in the two-dimensional subspace \cite{BBDLG}. Similarly, if $\alpha_1, \alpha_2, \alpha_3$, and $\alpha_4$ lie in the same 2-dimensional subspace $S$, then the origami set is dense in $S$. Note also that the density in one 2-dimensional subspace $S$ propagates through other 2-dimensional subspaces not defined by the angles of $S$ that leads to the original cluster points because of the definition of intersection points.
\end{proof}
\begin{example} Consider the set $$U=\{1, i, \frac{i-2}{\sqrt{2}}, \frac{2i-1}{\sqrt{5}}, j, \frac{j-1}{\sqrt{2}}, k, \frac{k-1}{\sqrt{2}}\}$$ of directions in $\R^4$.   Then, $M(U)$ is dense in $\R^4$ because the first four angles lie in the same plane. This angle set is obtained via the angle set for the Lipschitz order $\Z+\Z i + \Z j + \Z k$, by adding the single angle $\frac{2i-1}{\sqrt{5}}$. So in this case, simply adding a single angle to an angle set can change the origami construction from a lattice to a dense set. 

\end{example}

This is certainly not the only way for an origami set to be dense. Another way to do so given an origami lattice is to add an additional angle that will construct a point that lies between $0$ and $1.$

\begin{example} In 3 dimensions, we start with the angle set $U=\{1, i, j, \frac{1+i}{\sqrt{2}}, \frac{1+j}{\sqrt{2}}\}$ from Example \ref{cube}. Adding the angle $\sqrt{\frac{9}{22}}\left(\frac{2}{3}+i+j\right)$ to $U$ allows for the construction of the point $[\![0,1+i+j]\!]_{1, \sqrt{\frac{9}{22}}\left(\frac{2}{3}+i+j\right)}=\frac{1}{3},$ which can then be used to construct a sequence from $1$ to $0,$ making $M(U)$ a dense set. Note that $\sqrt{\frac{9}{22}}\left(\frac{2}{3}+i+j\right)$ does not lie in the same plane as any three angles from $U.$ This situation is illustrated in Figure \ref{third}. 

\begin{center}
\begin{figure}[h!]
\includegraphics[scale = .9]{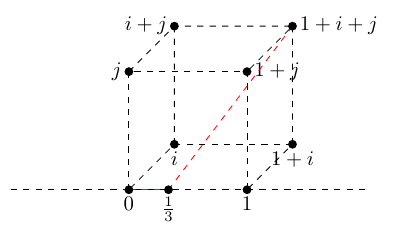}
\caption{The construction of $\frac{1}{3}$ in the origami construction $M(U)$.}
\label{third}
\end{figure}
\end{center}

\end{example}

\subsection{Irrelevant angles}
If $U=\{\alpha_1, \alpha_2, \dots, \alpha_m\}$ is a (finite) angle set such that $M(U)$ is a full lattice in $\R^n$, then there are actually infinitely many choices of angles that we may add to the angle set where the resulting set is the same lattice.

%One way to see this is through a cardinality argument. Let $S(U)$ be the projection of all the points in $M(U)$ and the lines going through all pairs of points onto the unit sphere in $\R^n$. Then $S(U)$ is a set compose of countably many arcs, so its cardinality is the same as $\R\times\Z$.  
%
%\as{Putting this proof here temporarily, some stuff needs to be reworked}
%

Let $U$ be a finite set of angles such that $M(U)$ is a lattice in $\R^n$. We will show that there exists an angle $\beta$ such that $M(U)=M(\widetilde{U})$, where $\widetilde{U}=U \cup \{\beta\}$. It suffices to show that for any $p\in M(U)$ and $\alpha \in U$, the line through $0$ at an angle of $\beta$ does not intersect the line through the point $p$ at an angle of $\alpha$.  That is, $$[\![p,0]\!]_{\alpha,\beta}=\emptyset.$$

Considering the set of all such lines that go through a point  $p \in M(U)$, at an angle in $\alpha \in U$. We will then project each of these lines onto the unit sphere in $\R^n$. We will show that there are infinitely many points remaining on the unit sphere that are not on the projection of any of the lines, and for each of those points, we can choose the angle corresponding to it and this angle will not create any new intersections when added to our angle set $U$.

\begin{lemma}\label{greatcircles}
The unit sphere $S^n$ cannot be written as a countable union of ``great circles" (that is, intersections of $S^n$ with $n$ dimensional hyperplanes in $\R^{n+1}$ containing the origin.)
\end{lemma}

\begin{proof}
We prove this by induction on $n$. 

Let $n=1$. Great circles in this case are, by definition, intersections of the unit circle with lines through the origin, i.e. pairs of antipodal points. A countable union of such sets clearly yields a countable set.

Now for the induction step, let $n>1$ and assume that the sphere $S^{n-1}$ cannot be written as the countable union of great circles. Assume, by way of contradiction, that $S^n$ can be written as the countable union of great circles, call this set $\mathscr{C}$. 

Recall we defined great circles to be the intersection of hyperplanes through the origin with $S^{n}$. Let $\mathcal{H}$ denote the set of hyperplanes which determine $\mathscr{C}$, that is $$\mathscr{C}=\{S^n\cap H : H \in \mathcal{H}\}.$$  Since $\R$ is uncountable, there are uncountably many hyperplanes in $\R^{n+1}$ (as $\R$-vector spaces). Let $K$ be a hyperplane such that $K\not\in \mathcal{H}$.  

Then there is a bijection between $K \cap S^{n}$ and $S^{n-1}$. Now,
\begin{eqnarray*}
K\cap S^n&=&K\cap \left(\bigcup_{C\in\mathscr{C}}C\right) \\
&= &K\cap \left(\bigcup_{H\in\mathcal{H}}(H\cap S^{n})\right)\\
&=& \bigcup_{H\in\mathcal{H}}(K\cap H\cap S^n)\\
\end{eqnarray*}

Notice $K\cap H\cap S^n=H\cap (K\cap S^n)$, and so under the bijection these are great circles of one lower dimension. In this way, we see that we obtain $S^{n-1}$ as a countable union of great circles, thus contradicting our induction hypothesis. 

%\as{It might be useful to do the 2-dimensional case in detail, since that is what I'm picturing in my head.}

\end{proof}

\begin{remark}
One could also prove Lemma \ref{greatcircles} by appealing to Baire's category theorem (see \cite{Munkres} for example), but we opted for a version that would be more accessible to all readers.
\end{remark}

\begin{proposition}\label{prop:irr}
Let $U=\{\alpha_1, \dots, \alpha_m\}$ be a finite set of distinct angles such that $M(U)$ is a full lattice in $\R^n$ containing 1, where $n \geq 3$. Then there are infinitely many choices of angle $\beta$ such that $M(\widetilde{U})=M(U)$, where $\widetilde{U}=U \cup \{\beta\}$.
\end{proposition}

\begin{proof}

Fix $0\neq p \in M(U)$ and $\alpha \in U$ so we consider the line $\ell$ through $p$ at an angle of $\alpha$. This line is the set of all points of the form $p+r\alpha$ for all $r \in \R$. There is a unique plane through the line $\ell$ and the origin. Since there are countably many points and directions, there are countably many such planes. As before, the intersections of these planes with $S^n$ are great circles. 

By Lemma \ref{greatcircles} we know there exists a point $q\in S^n$ that does not lie on any of the great circles, and thus does not lie on any of the planes. %, for then if we scale this point $q$ to be on the unit sphere, it still will not lie on any of the planes.% and hence it will not lie on the intersection of the sphere and one of the planes which is exactly the projection of the lines. 
Define $\beta$ to be the angle corresponding to the direction from the origin to $q$. Then, let $\widetilde{U}=U \cup \{\beta\}$. If $[\![p, 0]\!]_{\alpha, \beta}$ exists for some $p \in M(U)$ and $\alpha \in U$, then $p+r\alpha=s\beta$ for some $r,s \in \R$. But, $p+r\alpha$ is on one of the lines $\ell$ and hence cannot be a scalar multiple of $\beta$, a contradiction. Thus, $M(U)=M(\widetilde{U}).$
\end{proof}

The proof above is not constructive, so we conclude this paper with an example to demonstrate the construction of an ``irrelevant angle".

\begin{example} %\s{new example}

Let $U=\{\alpha_1, \alpha_2, \alpha_3, \alpha_4, \alpha_5\}=\{1, i, \frac{i-1}{\sqrt{2}}, j, \frac{j-i}{\sqrt{2}}\}$, so $M(U)=\Z+\Z i + \Z j$. Let $\beta=1+\pi i + \pi^2 j$ and scale appropriately so that $||\beta||=1$. If we define $\widetilde{U}\colonequals U \cup \{\beta\}$, then $M(U)=M(\widetilde{U})$ and hence $\beta$ is an ``irrelevant angle". In other words, for any point $p=a+bi+cj \in M(U)$ with $a,b,c \in \Z$, either $[\![p,0]\!]_{\alpha_i, \beta} \in M(U)$ or the intersection does not exist. We will show this for two intersections, as the remaining computations are very similar. It is important to note that the fact that $\pi$ is transcendental will help with the remaining intersections.

If $[\![p, 0]\!]_{\alpha_1, \beta}$ exists, then $a+bi+cj+r(1)=s(1+\pi i + \pi^2 j)$ for some $r,s \in \R$. Then, $b+\pi s=0$ so $s=-\frac{b}{\pi}$. We also have $c+\pi^2 s=0$ so $s=-\frac{c}{\pi^2}$. This means that $\frac{b}{\pi}=\frac{c}{\pi^2}$ and so $\pi b=c$ unless $b=c=0$. If $b,c \neq 0$, we get $\pi=\frac{c}{b}$ which cannot be true since $\pi$ is irrational. Therefore the intersection only exists when $b=c=0$ which forces $s=0$ and hence $[\![p, 0]\!]_{\alpha_1, \beta}=0 \in M(U)$.

If $[\![p, 0]\!]_{\alpha_2, \beta}$ exists, then $a+bi+cj+r(i-1)=s(1+\pi i + \pi^2 j)$ for some $r,s \in \R$. Then, $c=\pi^2 s$ so $s=\frac{c}{\pi^2}$. We also have $b+r=s\pi =\frac{c}{\pi}$ so $r=\frac{c}{\pi}-b$. Similarly, $a-r=s$ so $r=a-s=a-\frac{c}{\pi^2}$. Then, $\frac{c}{\pi}-b=a-\frac{c}{\pi^2}$. Multiplying both sides by $\pi^2$, we have $c\pi-b\pi^2=a\pi^2-c$ so $(a+b)\pi^2-c\pi-c=0$. This can only be true if $a+b=0$ and $c=0$ since $a,b,c \in \Z$ and $\pi$ is transcendental. If $c=0$ then $s=0$ and $[\![p, 0]\!]_{\alpha_2, \beta}=0 \in M(U)$. 

A very similar argument holds for the remaining intersections $[\![p, 0]\!]_{\alpha_3, \beta}$ $[\![p, 0]\!]_{\alpha_4, \beta}$, and $[\![p, 0]\!]_{\alpha_5, \beta}$.

\end{example}

In fact, for an angle set $U=\{\alpha_1, \alpha_2, \dots, \alpha_m\}$ such that $M(U)$ is a lattice in $\R^n$, we can construct an angle $\beta$ such that $M(U)=M(\widetilde{U})$ where $\widetilde{U}=U \cup \{\beta\}$ by choosing a transcendental number $\eta$ carefully and defining $$\beta \colonequals \alpha_1+\eta \alpha_2+\eta^2 \alpha_3 + \cdots +\eta^{m-1} \alpha_m$$ and scaling appropriately.

\subsection*{Acknowledgements}
The work described here stemmed from an Honors undergraduate thesis by Deveena Banerjee and advised by Adriana Salerno. Sara Chari was a member of the Honors panel, and the following year the three authors worked on extending Deveena's thesis work. For this, we would primarily like to thank the Bates College math department for their support, encouragement, and guidance. We would also like to give our thanks to the other two panelists, John Voight and Geneva Laurita, for their insightful comments and suggestions. Finally, we would like to thank the referee, whose suggestions vastly improved our paper, and in particular the proof of Proposition \ref{prop:irr}.

\end{document}